\documentclass[11pt]{article}
\usepackage{amssymb}
\usepackage[perpage,symbol]{footmisc}
\usepackage{listings}
\usepackage{amsfonts}
\usepackage{enumerate}
\usepackage{amsmath}
\usepackage{cite}
\usepackage{array}
\usepackage{booktabs}

\begin{document}

\newtheorem{theorem}{Theorem}[section]
\newtheorem{corollary}[theorem]{Corollary}
\newtheorem{definition}[theorem]{Definition}
\newtheorem{proposition}[theorem]{Proposition}
\newtheorem{lemma}[theorem]{Lemma}
\newtheorem{conjecture}[theorem]{Conjecture}
\newenvironment{proof}{\noindent {\bf Proof.}}{\rule{3mm}{3mm}\par\medskip}
\newcommand{\remark}{\medskip\par\noindent {\bf Remark.~~}}

\title{Incidence Matrices of  Finite Quadratic Spaces\footnote{This work is supported by the National Natural Science Foundation of China (No. 11071160).}}
\author{Chunlei Liu\footnote{Dept. of Math., Shanghai Jiaotong Univ., Shanghai,
200240, clliu@sjtu.edu.cn.}, Yan Liu\footnote{Corresponding author,
Dept. of Math., SJTU, Shanghai, 200240, liuyan0916@sjtu.edu.cn.}}

\maketitle
\thispagestyle{empty}

\abstract{In this paper, we first prove the 2-rank of full incidence
matrix of $PG(n,q)$ with $q$ an odd prime power. Then by the
quadratic form defined on $PG(n,q)$, the points of it are classified
as isotropic and anisotropic points. Under this classification,  we
divide the full incidence matrix into four sub-matrices. Then by
using the software package $Magma$, we give a general conjecture for
the 2-rank of sub-matrices of the full incidence matrix in $PG(n,q)$
and  prove it in the case of $n=1, 2$. }


\section{\small{INTRODUCTION}}
Throughout this paper, $\mathbb{F}_{q}$ denotes the finite field of order $q$, where $q$  is an odd
prime power. And we refer projective geometry to finite projective geometry.
Projective geometry is formally defined as a collection of objects with incidence relation between them. We use $PG(n, q)$ to denote the $n$-dimensional classical projective geometry of order $q$. Due to a theorem of Veblen and Young \cite{16}, we can construct $PG(n, q)$  using a $(n+1)$-dimensional vector space $V$ over $\mathbb{F}_{q}$. In this model the $1$-dimensional subspaces of $V$ represent points, the $2$-dimensional subspaces represent lines, the $3$-dimensional subspaces represent planes,  and the $n$-dimensional subspaces represent hyperplanes. The incidence is given by the natural containment. The $(n+1)$-dimensional vector space $V$ is also called the underlying vector space of $PG(n, q)$. Given a basis of the underlying vector space $V$, we have $V \cong (\mathbb{F}_{q})^{n+1}$. So in the following, we do not distinguish between $V$ and $(\mathbb{F}_{q})^{n+1}$. Since scalar multiples of a nonzero vector generate the same 1-dimensional subspace, we often ``normalize" our vectors to ensure the uniqueness of the representation. This means that the first nonzero coordinate of every vector is 1.
We define a quadratic form $Q$ on $V$: $Q(X)=x_{0}^{2}-x_{1}^{2}+x_{2}^{2}+\cdots +(-1)^{n}\alpha x_{n}^{2}$, where $\alpha$ is a nonzero element of $\mathbb{F}_{q}$, for any $X \in V$. Then the associated bilinear form over $V$ is: $\langle X, Y \rangle=\frac{1}{2}\{Q(X+Y)-Q(X)-Q(Y)\}$, for any $X, Y \in V$. It is not hard to see that the quadratic form defined above is nondegenerate.

From the above illustration, we can use ``orthogonal complement" to represent hyperplanes. That is, the set of vectors orthogonal to a $n$-dimensional subspace forms a $1$-dimensional subspace. So there exists a bijection between the set of all hyperplanes and the set of all points of $PG(n, q)$, moreover, each hyperplane can also be represented by a nonzero vector. For example, let $P$ be a point in $PG(n,q)$, then $P^{\perp}$ denote the corresponding hyperplane which is orthogonal to $P$. And a point $R$ is on the hyperplane $P^{\perp}$ if and only if $\langle P,~R\rangle=0$. Moreover, because the bilinear form is symmetric, $R \in P^{\perp} \Longleftrightarrow P \in R^{\perp}$.

Let $A$ be the $(q^{n}+q^{n-1}+\cdots+q+1)\times(q^{n}+q^{n-1}+\cdots+q+1)$ incidence matrix of $PG(n,q)$, the rows of which are labeled by the points in $PG(n,q)$ and the columns are labeled by the corresponding hyperplanes. Whenever a point lies on a hyperplane, the element of the matrix at the intersection of those two labels is a 1, and otherwise, it is a 0.
The 2-rank of $A$ can be obtained by the following proposition:
\begin{proposition}\label{Prop:1.1}
The $2$-rank of $\mathbf{A}$ is $q^{n}+q^{n-1}+\cdots+q+1$ or $q^{n}+q^{n-1}+\cdots+q$ according as $n$ is odd or even.
\end{proposition}
\begin{proof}
The result is obviously when $n=1$. We only consider the cases when $n \geq 2$ in the following.
Suppose that $Y$, $Z$ are two distinct points on $PG(n,q)$, then the number of hyperplanes through them is equal to the number of points satisfying the following equations:
\[
  \left\{
   \begin{aligned}
   x_{0}y_{0}-x_{1}y_{1}+\cdots+(-1)^{n}\alpha x_{n}y_{n}&=0  \\
   x_{0}z_{0}-x_{1}z_{1}+\cdots+(-1)^{n}\alpha x_{n}z_{n}&=0 \\
   \end{aligned}
   \right.
\]
Easy calculation shows that the number of hyperplanes through 2 fixed points is equal to $q^{n-2}+q^{n-3}+\cdots+q+1$. One can immediately obtain a symmetric $2$-design $T$ from $PG(n,q)$, where the points of $T$ are the points of $PG(n,q)$ and the blocks of $T$ are the hyperplanes of $PG(n,q)$. Then the result can be obtained by \cite{14}(p.309).
\end{proof}

Next, we give a partition of the points of $PG(n,q)$ by the bilinear form defined above. Let $P = (x_{0}, x_{1}, \cdots,  x_{n})$ be a point in $PG(n,q)$, then $P$ is called an {\it isotropic, square anisotropic or non-square anisotropic} point if $\langle P, P \rangle = Q(P)$ is a zero, nonzero square or non-square element of $\mathbb{F}_{q}$. According to this partition, we divide $A$ into nine sub-matrices:
\begin{equation}
\left( \begin{array}{ccc}
\mathbf{\dot{A}}_{11}& \mathbf{\dot{A}}_{12}&\mathbf{\dot{A}}_{13}\\
\mathbf{\dot{A}}_{21}& \mathbf{\dot{A}}_{22}&\mathbf{\dot{A}}_{23}\\
\mathbf{\dot{A}}_{31}& \mathbf{\dot{A}}_{32}&\mathbf{\dot{A}}_{33}\\
\end{array} \right),
\end{equation}
where the columns of $\mathbf{\dot{A}}_{11}$, $\mathbf{\dot{A}}_{12}$ and $\mathbf{\dot{A}}_{13}$ are labeled by the square anisotropic, nonsquare anisotropic, isotropic points respectively, and the rows of $\mathbf{\dot{A}}_{11}$, $\mathbf{\dot{A}}_{21}$ and $\mathbf{\dot{A}}_{31}$ are labeled by their corresponding hyperplanes. When $n=2$, according to \cite{3,4,5,6} we have known the 2-rank of any one of the nine submatrix of $A$ under the quadratic form $Q'(X)=\alpha x_{1}^{2}-x_{0}x_{2}$ which is equivalent to $Q(X)=x_{0}^{2}-x_{1}^{2}+\alpha x_{2}^{2}$, where $\alpha$ is a nonzero square element of $\mathbb{F}_{q}$.
\begin{proposition}[\cite{4}]
\[Rank(\mathbf{\dot{A}}_{11})=\begin{cases}
\frac{q^{2}-1}{4}+q-1,& if~q \equiv 1~(mod~4),\\
\frac{q^{2}-1}{4}+q+1,& if~q \equiv 3~(mod~4).
\end{cases}\]
\end{proposition}
\begin{proposition}[\cite{6}]
\[Rank(\mathbf{\dot{A}}_{12})=Rank(\mathbf{\dot{A}}_{21})=\begin{cases}
\frac{(q-1)^{2}}{4}+q,& if~q \equiv 1~(mod~4),\\
\frac{(q-1)^{2}}{4}+q-1,& if~q \equiv 3~(mod~4).
\end{cases}\]
\end{proposition}
\begin{proposition}[\cite{5}]
\[Rank(\mathbf{\dot{A}}_{22})=\frac{q^{2}-1}{4}\]
\end{proposition}
\begin{proposition}[\cite{3}]
The $2$-rank of $\mathbf{\dot{A}}_{13}(\mathbf{\dot{A}}_{31})$ and $\mathbf{\dot{A}}_{33}$ are $q$ and $q+1$, respectively. Moreover, $\mathbf{\dot{A}}_{23}(\mathbf{\dot{A}}_{32})$ is a zero matrix.
\end{proposition}
In the rest of this paper, we only consider the following partition of $A$:
\begin{equation}\label{eq:2}
\left( \begin{array}{cc}
\mathbf{A}_{11}& \mathbf{A}_{12}\\
\mathbf{A}_{21}& \mathbf{A}_{22}\\
\end{array} \right),
\end{equation}
where the columns of $\mathbf{A}_{11}$, $\mathbf{A}_{12}$ are labeled by the anisotropic, isotropic points respectively, and the rows of $\mathbf{A}_{11}$, $\mathbf{A}_{21}$ are labeled by their corresponding hyperplanes. It is obviously that in this two partitions $\mathbf{A}_{22}=\mathbf{\dot{A}}_{33}$. Again from \cite{3} we have known the 2-rank of $\mathbf{A}_{12}$ and $\mathbf{A}_{21}$ under the quadratic form $Q'(X)=\alpha x_{1}^{2}-x_{0}x_{2}$ with $\alpha$ a nonzero square element of $\mathbb{F}_{q}$ in the case of $n=2$.
\begin{proposition}[\cite{3}]
\[Rank(\mathbf{A}_{12})=Rank(\mathbf{A}_{21})=q,~~Rank(\mathbf{A}_{22})=q+1.\]
\end{proposition}
The unresolved cases when $n=2$ under this partition will be discussed in Section 3.
Moreover, for the four sub-matrices coming from this partition, we have a general conjecture:
\begin{conjecture}\label{Con:1.7}
In $PG(n,q)$, the $2$-rank of sub-matrices coming from the first partition of $\mathbf{A}$ under the quadratic form with $\alpha$ a nonzero square or non-square element of $\mathbb{F}_{q}$ are as follows:
\begin{enumerate}
   \item
    \begin{enumerate}[$(a)$]
    \item when $n$ is odd, $\mathbf{A}_{11}$ is full rank.\\
    \item when $n$ is even, the $2$-rank of $\mathbf{A}_{11}$ is one less than its order.
    \end{enumerate}
   \item $\mathbf{A}_{22}$ is always full rank.\\
   \item  \begin{enumerate}[$(a)$]
           \item when $n \geq 3$ is odd and , the 2-rank of $\mathbf{A}_{12}(\mathbf{A}_{21})$ is
            \begin{enumerate}[$(i)$]
              \item $q^{n-1}+q^{n-3}$, when $\alpha$ is a nonzero square element;\\
              \item $q^{n-1}+q^{n-3}+\cdots+q^{2}$, when $\alpha$ is a non-square element.
            \end{enumerate}
          \item when $n$ is even, the 2-rank of $\mathbf{A}_{12}(\mathbf{A}_{21})$ is $q^{n-1}+q^{n-3}+\cdots+q$.
         \end{enumerate}

  \end{enumerate}
\end{conjecture}
The conjecture in the case of $n=1$ and $2$ will be confirmed in Section 2 and 3, which is described as the following theorem. When $n \geq 3$, we can not give proof. However, by use of Magma, the results for some fixed prime power $q$ can be used to verify the reasonability of this conjecture.
\begin{theorem}For any order $q$,
\begin{enumerate}
\item in $PG(1,q)$, $\mathbf{A}_{11}$ and $\mathbf{A}_{22}$ are both full-rank;\\
\item in $PG(2,q)$, $Rank(\mathbf{A}_{11})=q^{2}-1,Rank(\mathbf{A}_{12})=Rank(\mathbf{A}_{21})=q$ and $\mathbf{A}_{22}$ is full rank.
\end{enumerate}
\end{theorem}
To get the size of each sub-matrix of $A$ easily, from \cite{1}, we first give a general formula for the number of isotropic points:
\begin{lemma}\label{Le:1.8}
For any projective geometry $PG(n,q)$,
\begin{enumerate}
\item when $n$ is odd, the number of isotropic points is
  \begin{enumerate}[$(a)$]
  \item $q^{n-1}+q^{n-2}+\cdots+q+1+q^{\frac{n-1}{2}}$ when $\alpha$ is a nonzero square element of $\mathbb{F}_{q}$;\\
  \item $q^{n-1}+q^{n-2}+\cdots+q+1-q^{\frac{n-1}{2}}$ when $\alpha$ is a non-square element of $\mathbb{F}_{q}$.
  \end{enumerate}
\item when $n$ is even, the number of isotropic points is $q^{n-1}+q^{n-2}+\cdots+q+1$.
\end{enumerate}
\end{lemma}
$\mathbf{Acknowledgement}$: This research is motivated by a talk
given by Qing Xiang at SJTU.

\section{\small{THE RANK OF SUB-MATRICES IN THE CASE OF $PG(1,q)$}}
As we did in Section 1, we have a partition of the full incidence matrix $A$ of $PG(1,q)$. In the following, we will give the 2-rank of each sub-matrix under this partition. Since from Lemma \ref{Le:1.8}, when $\alpha$ is a non-square element of $\mathbb{F}_{q}$, the number of isotropic points is zero. So $\mathbf{A}_{11} = \mathbf{A}$ (the full incidence matrix of $PG(1,q)$), which is full rank by Proposition \ref{Prop:1.1}, then we only need to consider the other cases.
\begin{proposition}
As in $(\ref{eq:2})$, we can get four sub-matrices of the incidence matrix $\mathbf{A}$ of $PG(1,q)$. Under the quadratic form $Q$ with $\alpha$ a nonzero square element of $\mathbb{F}_{q}$, $\mathbf{A}_{11}$ and $\mathbf{A}_{22}$ are both full-rank.

\end{proposition}
\begin{proof}
We claim that if $P=[p_{0}:~p_{1}]$ is a nonzero point of $PG(1,q)$,  then $P^{\perp}$ is a nonzero point. Otherwise, if $P^{\perp}=R=[r_{0}:~r_{1}]$ is a zero point, then $p_{0}=\pm p_{1}$. In can be found that $P$ is a zero point by the definition, which is a contradiction. Similarly, we can prove that if $P$ is a zero point of $PG(1,q)$,  then $P^{\perp}$ is also a zero point. It is easy to see that $\mathbf{A}_{11}$ and $\mathbf{A}_{22}$ are both full-rank.
\end{proof}
To sum up, we have that Conjecture \ref{Con:1.7} is true when $n=1$.
\section{\small{THE RANK OF SUB-MATRICES IN THE CASE OF $PG(2,q)$}}
In $PG(2,q)$, a hyperplane is in fact a line. As in \cite{4}, we consider the problem under the quadratic form $Q'(X)=\alpha x_{1}^{2}-x_{0}x_{2}$ which is equivalent to $Q(X)=x_{0}^{2}-x_{1}^{2}+\alpha x_{2}^{2}$, where $\alpha$ is a nonzero element of $\mathbb{F}_{q}$. In $PG(2,q)$, there is a very important property which is described as follows:
\begin{lemma}
The rank of any sub-matrix constructed  from the partition of $\mathbf{A}$ is independent of $\alpha$.
\end{lemma}
This lemma can be easily proved by the incidence relationship between various types of points and lines and the number of them.
So to avoid duplication, in the case of $PG(2,q)$, we will always assume that $\alpha$ is a nonzero square, in other words, we only consider the problem under the quadratic form  $Q'(X)=x_{1}^{2}-x_{0}x_{2}$.

 Following the notations in \cite{4}, we refer the {\it external}, {\it internal}  and {\it absolute}  points in this section to square anisotropic, non-square anisotropic and isotropic points; their corresponding lines are called {\it secant, passant} and {\it tangent} lines. Let $E$ denote the set of all external points. To prove the proposition following, we first introduce a lemma which can be obtained by simple counting and you can see \cite{2} for more details and related results.
\begin{lemma}
[\cite{2} p.178] We have the following tables:
\begin{table}[htbp]
\caption{Number of points on lines of various types}
\centering
\begin{tabular}{cccc}
 \toprule
 Name&Absolute points& External points& Internal points\\
 \midrule
 Tangent lines&$1$&$q$&0\\
 Secant lines&$2$&$\frac{q-1}{2}$&$\frac{q-1}{2}$\\
 Passant lines&$0$&$\frac{q+1}{2}$&$\frac{q+1}{2}$\\
 \bottomrule
\end{tabular}
 \end{table}
 \begin{table}[htbp]
\caption{Number of lines through points of various types}
\centering
\begin{tabular}{cccc}
 \toprule
 Name&Tangent lines&Secant lines &Passant lines \\
 \midrule
 Absolute points&$1$&$q$&0\\
 External points&$2$&$\frac{q-1}{2}$&$\frac{q-1}{2}$\\
 Internal points&$0$&$\frac{q+1}{2}$&$\frac{q+1}{2}$\\
 \bottomrule
\end{tabular}
 \end{table}
 \end{lemma}
 In order to state the proof of Lemma \ref{Le:3.4} explicitly, we give a definition.
\begin{definition} Let $P \in E$. We define
$N_{E}(P)=\{Q \in E | Q \in l, l \in Se_{P}\cup Pa_{P}\}\backslash\{P\}$, where $Se_{P}$ $($ respectively, $Pa_{P}$$)$ denotes the set of secant $($respectively,  passant$)$ lines through $P$. That is $N_{E}(P)$ is the set of external points on the secant or passant lines through $P$ with $P$  excluded.
\end{definition}
Using the software package {\it Magma}\cite{7} we are able to determine the rank of $\mathbf{A}_{11}$ with various orders $q$ $($Table $1)$.
\begin{table}[htbp]
\caption{Rank of $\mathbf{A}_{11}$}
\centering
\begin{tabular}{lccccc}
 \hline
 q& 3&5 &7& 9& 11\\
 \hline
 Dim & 8&24 &48& 80& 120\\
 \hline
\end{tabular}
\end{table}

\noindent According to this data, we have the following proposition:
\begin{proposition}\label{Pro:3.3}
The $2$-rank of $\mathbf{A}_{11}$ in the case of $PG(2,q)$ is $q^{2}-1$.
\end{proposition}

\begin{lemma}\label{Le:3.4}
If we view $\mathbf{A}_{11}$ as a matrix over $F$, then $(\mathbf{A}_{11})^{4}= \mathbf{J}-\mathbf{I}$, where $\mathbf{J}$ is the all-one square matrix and $\mathbf{I}$ is the identity matrix.
\end{lemma}
\begin{proof} Firstly, we divide $\mathbf{A}_{11}$ into 4 sub-matrices:
\begin{equation}
\left( \begin{array}{cc}
\mathbf{\dot{A}}_{11}& \mathbf{\dot{A}}_{12}\\
\mathbf{\dot{A}}_{21}& \mathbf{\dot{A}}_{22}\\
\end{array} \right),
\end{equation}
where the columns of $\mathbf{\dot{A}}_{11}$ and $\mathbf{\dot{A}}_{12}$ are labeled by the external and internal points respectively, and the rows of $\mathbf{\dot{A}}_{11}$ and $\mathbf{\dot{A}}_{21}$ are labeled by secant and passant lines. From the table above, there exists no internal points on tangent line, i.e., any internal point is on some secant line or some passant line. Then
\[(\mathbf{A}_{11})^2=\left( \begin{array}{cc}
\mathbf{M}_{1}& \mathbf{M}_{2}\\
\mathbf{M}_{3}&\mathbf{M}_{4}\\
\end{array} \right),\]
where $M_{2}$ ( respectively, $M_{3}$ ) is the $\frac{q(q+1)}{2} \times \frac{q(q-1)}{2}$ ( respectively, $\frac{q(q-1)}{2} \times \frac{q(q+1)}{2}$) all-one matrix and $M_{4}$ is a square matrix of order $\frac{q(q-1)}{2}$ whose diagonal entries are $0$ and the other are $1$; the row of $M_{1}$ indexed by the point $P_{i}$ can be viewed as the characteristic vector of $N_{E}(P_{i})$ with respect to $E$.
Moreover, by the multiplication of the block matrix,
\begin{equation}\label{Eq:4.5}
(\mathbf{A}_{11})^4=(\mathbf{A}_{11})^2 \cdot (\mathbf{A}_{11})^2=
\left( \begin{array}{cc}
(\mathbf{M}_{1})^{2}+\mathbf{M}_{2}\cdot\mathbf{M}_{3}& \mathbf{M}_{1}\cdot\mathbf{M}_{2}+\mathbf{M}_{2}\cdot\mathbf{M}_{4}\\
\mathbf{M}_{3}\cdot\mathbf{M}_{1}+\mathbf{M}_{4}\cdot\mathbf{M}_{3}&\mathbf{M}_{3}\cdot\mathbf{M}_{2}+(\mathbf{M}_{4})^{2}\\
\end{array} \right).
\end{equation}
Next we will calculate each part of the block matrix by the relationship between the points and the lines in $PG(2,q)$.

The entry in the $i^{th}$ row and $j^{th}$ column of $(\mathbf{M}_{1})^{2}$ is given by $|N_{E}(P_{i})\cap N_{E}(P_{j})|~ (mod ~2)$. When $i=j$, we consider the complementary set $N_{E}(P_{i})^{c}$ of $N_{E}(P_{i})$ with respect to $E$, whose cardinality is turned out to be $2(q-1)+1$. So the parity of $|N_{E}(P_{i})|$ is the same as the parity of $|E|-2(q-1)-1= \frac{q(q+1)}{2}-1$. When $i \neq j$, we consider the complementary set $\big(N_{E}(P_{i})\cap N_{E}(P_{j})\big)^{c}$ of $N_{E}(P_{i})\cap N_{E}(P_{j})$ with respect to $E$.

\noindent Case $\mathbf{I}$. The line through $P_{i}$ and $P_{j}$ is a tangent line to $\mathcal {O}$. In this case, $|\big(N_{E}(P_{i})\cap N_{E}(P_{j})\big)^{c}|=3(q-2)+3=3(q-1)$, which is even.

\noindent Case $\mathbf{II}$. The line through $P_{i}$ and $P_{j}$ is not a tangent line to $\mathcal {O}$. In this case, $|\big(N_{E}(P_{i})\cap N_{E}(P_{j})\big)^{c}|=4(q-3)+6 $, which is also even.

So the parity of $|N_{E}(P_{i})\cap N_{E}(P_{j})|$ is the same as that of $|E|$. From the results calculated above, we have \[(\mathbf{M}_{1})^{2}=\begin{cases}
\mathbf{J}-\mathbf{I},& if~q \equiv 1~(mod~4),\\
\mathbf{I},& if~q \equiv 3~(mod~4),
\end{cases}\]
The other matrices which can be easily calculated are as follows:
\begin{displaymath}
\mathbf{M}_{2}\cdot\mathbf{M}_{3}=\begin{cases}
\mathbf{0},& if~q \equiv 1~(mod~4),\\
\mathbf{J},& if~q \equiv 3~(mod~4),
\end{cases}
\mathbf{M}_{1}\cdot\mathbf{M}_{2}=\begin{cases}
\mathbf{0},& if~q \equiv 1~(mod~4),\\
\mathbf{M}_{2},& if~q \equiv 3~(mod~4),
\end{cases}
\end{displaymath}
\begin{displaymath}
\mathbf{M}_{2}\cdot\mathbf{M}_{4}=\begin{cases}
\mathbf{M}_{2},& if~q \equiv 1~(mod~4),\\
\mathbf{0},& if~q \equiv 3~(mod~4),
\end{cases}
\mathbf{M}_{3}\cdot\mathbf{M}_{1}=\begin{cases}
\mathbf{0},& if~q \equiv 1~(mod~4),\\
\mathbf{M}_{3},& if~q \equiv 3~(mod~4),
\end{cases}
\end{displaymath}
\begin{displaymath}
\mathbf{M}_{4}\cdot\mathbf{M}_{3}=\begin{cases}
\mathbf{M}_{3},& if~q \equiv 1~(mod~4),\\
\mathbf{0},& if~q \equiv 3~(mod~4),
\end{cases}
\mathbf{M}_{3}\cdot\mathbf{M}_{2}=\begin{cases}
\mathbf{J},& if~q \equiv 1~(mod~4),\\
\mathbf{0},& if~q \equiv 3~(mod~4),
\end{cases}
\end{displaymath}

\[\mathbf{M}_{4}^{2}=\begin{cases}
\mathbf{I},& if~q \equiv 1~(mod~4),\\
\mathbf{J}-\mathbf{I},& if~q \equiv 3~(mod~4).
\end{cases}\]
Then the proposition can be proved by substituting the results above for the block in equation (\ref{Eq:4.5}).
\end{proof}
\begin{corollary}
If we view $\mathbf{A}_{11}$ as a matrix over $F$, then $(\mathbf{A}_{11})^{5}= \mathbf{A}_{11}$.
\end{corollary}

\noindent $\mathbf{Proof}$ $\mathbf{of}$ $\mathbf{Proposition}$ \ref{Pro:3.3}.
Using the same method in \cite{3}, we have that the $F$-null space of $\mathbf{A}_{11}$ is equal to the span of the rows of $(\mathbf{A}_{11})^{4}+ \mathbf{I}$, where $\mathbf{I}$ is the identity matrix. By Lemma \ref{Le:3.4}, $(\mathbf{A}_{11})^{4}+ \mathbf{I}= \mathbf{J}$. Hence, the dimension of the $F$-null space of $\mathbf{A}_{11}$ is $1$, then the proposition is proved.

\noindent In addition to this, by the detailed information supplied by \cite{3}, we can have that $Rank(\mathbf{A}_{12})=Rank(\mathbf{A}_{21})=q$ and $\mathbf{A}_{22}$ is full-rank by easy calculation. Until now, together with Section 2, we can say that Conjecture \ref{Con:1.7} is true when $n=1, 2$.

\section{\small{APPENDICES}}
In this section, we only list the $Magma$ commands in the case of $PG(3,q)$
 under the quadratic form with $\alpha$ a non-square element in $\mathbb{F}_{q}$. Besides, we give a example for how to get the 2-rank of $\mathbf{A}_{11}$. The commands
 under the quadratic form with $\alpha$ a nonzero square element in $\mathbb{F}_{q}$ can be simply obtained by deleting $\alpha$. Moveover, you can get the 2-rank of any submatrix of the full incidence matrix in the case of $PG(n,q)$ by simply modifying the parameters in the command statement.
\lstset{basicstyle=\footnotesize, frame=shadowbox,xleftmargin=1.5em,xrightmargin=1.5em,aboveskip=1em, belowskip=1em}

\noindent {\it To create the set of points of $PG(3,q)$:}
\begin{lstlisting}
K := GF(q);
a :=PrimitiveElement(K);
vs := VectorSpace(K, 4);
pt := Setseq({Normalize(v): v in vs | v ne vs!0});
\end{lstlisting}
{\it To create the various types of points:}
\begin{lstlisting}
pt1 := Setseq({v: v in pt| IsSquare(v[1]^2-v[2]^2+v[3]^2
    -a*v[4]^2)and (v[1]^2-v[2]^2+v[3]^2-a*v[4]^2 ne 0)});
pt2 := Setseq({v: v in pt| not IsSquare(v[1]^2-v[2]^2
    +v[3]^2-a*v[4]^2)});
pt0 := Setseq({v: v in pt|v[1]^2-v[2]^2+v[3]^2-a*v[4]^2 eq 0});
pt~0 := Setseq({v: v in pt| v[1]^2-v[2]^2+v[3]^2-a*v[4]^2 ne 0});
\end{lstlisting}
{\it To create the various types of hyperplanes:}

\begin{lstlisting}
Plan1 := {({w : w in pt | w[1]*pt1[i][1]-w[2]*pt1[i][2]
      +w[3]*pt1[i][3]-a*w[4]*pt1[i][4] eq 0}): i in {1..#pt1}};
Plan2 := {({w : w in pt | w[1]*pt2[i][1]-w[2]*pt2[i][2]
      +w[3]*pt2[i][3]-a*w[4]*pt2[i][4] eq 0}): i in {1..#pt2}};
Plan0 :={({w : w in pt | w[1]*pt0[i][1]-w[2]*pt0[i][2]
      +w[3]*pt0[i][3]-a*w[4]*pt0[i][4] eq 0}): i in {1..#pt0}};
Plan~0 := {({w : w in pt | w[1]*pt~0[i][1]-w[2]*pt~0[i][2]
       +w[3]*pt~0[i][3]-a*w[4]*pt~0[i][4]eq 0}):i in {1..#pt~0}};
\end{lstlisting}
{\it To get the 2-rank of $\mathbf{A}_{11}$:}
\begin{lstlisting}
B :={ {i : i in {1..#pt1}| pt1[i] in Pl}: Pl in Plan1};
D := IncidenceStructure<#P|B>;
M := ChangeRing(IncidenceMatrix(D), GF(2));
Rank(M);
\end{lstlisting}

\end{document}